\documentclass[11pt,a4paper]{article}

\usepackage{amssymb, epsfig}

\usepackage[a4paper, top=1in, bottom=1in, left=1in, right=1in ]{geometry}

\usepackage{latexsym}
\usepackage{amsmath}
\usepackage{mathtools}
\usepackage{setspace}
\usepackage{eepic}
\usepackage{amsmath}
\usepackage{amsthm}
\usepackage{makeidx}
\usepackage{graphicx}
\usepackage{stmaryrd}
\usepackage{bbding}
\usepackage{float}
\usepackage[latin1]{inputenc}  
\usepackage[english]{babel}
\usepackage{amssymb}
\usepackage{stmaryrd}
\usepackage{fancybox} 
\usepackage{fancyhdr}
\usepackage{extarrows}
\usepackage[usenames,dvipsnames]{color}
\usepackage[all]{xy}
\usepackage{pdflscape}
\usepackage{pinlabel}
 \usepackage{cite}
\usepackage{filecontents}
\usepackage{extarrows}
\usepackage{upgreek}
\usepackage{stmaryrd}
\usepackage{mathtools}
\usepackage[usenames,dvipsnames]{color}
\usepackage[all]{xy}

\usepackage{hyperref}
 
\hypersetup{
bookmarks=true,         
    unicode=false,          
    pdftoolbar=true,        
    pdfmenubar=true,        
    pdffitwindow=true,     
    pdftitle={Hopf algebroids associated to Jacobi algebras},    
    pdfauthor={Ana Rovi},     
    pdfsubject={second paper},   
    pdfcreator={Creator},   
    pdfproducer={Producer}, 
    pdfkeywords={Lie--Rinehart algebras} {Hopf algebroids} {Leibniz algebras} {right connections} {right modules}, 
    pdfnewwindow=true,      
    colorlinks=true,       
    linkcolor=MidnightBlue,          
    citecolor=RoyalBlue,        
    filecolor=magenta,      
    urlcolor=black           
}

\title{Hopf algebroids associated to Jacobi algebras}
 \date{}
 \author{Ana Rovi \footnote{Research funded by an EPSRC DTA grant}\\
School of Mathematics and Statistics\\ 
University of Glasgow\\  
Glasgow G12 8QW, UK \\
\begin{small}{ (\href{a.rovi.1@research.gla.ac.uk}{a.rovi.1@research.gla.ac.uk}) }\end{small}
}

\numberwithin{equation}{section}

\newtheorem{theorem}{Theorem}[section]

\newtheorem{lemma}[theorem]{Lemma}
\newtheorem{remark}[theorem]{Remark}

\newtheorem{proposition}[theorem]{Proposition}
\newtheorem{definition}[theorem]{Definition}
\newtheorem{corollary}[theorem]{Corollary}

\newcommand{\dr}{\delta^\nabla_r}

\newcommand{\Cr}{\mathcal C^\nabla_r}

\newcommand{\dt}{\mathfrak D}

\begin{document}

\setcounter{secnumdepth}{3}
\setcounter{tocdepth}{3}

\maketitle

\begin{center}
\begin{footnotesize}
{\it For my father}
\end{footnotesize}
\end{center}

\begin{abstract}
We give examples of Lie--Rinehart algebras whose universal enveloping algebra is  not a  Hopf algebroid either in the sense of B\"ohm and Szlach\'anyi or in the sense of Lu. We construct these examples as quotients of a canonical Lie--Rinehart algebra over a Jacobi algebra which does admit an antipode. 
\end{abstract}


\singlespacing

\section{Introduction}

 {Hopf algebroids} are generalizations of Hopf algebras. In the literature,  the term \emph{Hopf algebroid}  refers to either of the following three concepts:  \emph{Hopf algebroids} in the sense of B\"ohm and Szlach\'anyi \cite{GabiK} (which we will call \emph{full}), to Hopf algebroids in the sense of Lu \cite{JHLu}, or to \emph{left Hopf algebroids} (which were introduced under the name $\times_R$-Hopf algebras by  Schauenburg \cite{Schauenburg}).

In \cite{GabiK}, an example of a full Hopf algebroid is given that does not satisfy the axioms of \cite{JHLu}; it is however unknown whether all Hopf algebroids in the sense of \cite{JHLu} are full. 
 Moreover, until  \cite{KR1} it had been an open question whether  Hopf algebroids, in the sense of \cite{GabiK} or in the sense of \cite{JHLu}, were  equivalent to left Hopf algebroids. It turns out that the universal enveloping algebra of a Lie--Rinehart algebra (a fundamental example of a left Hopf algebroid, see  \cite[Example 2]{KOKR}) will not carry an antipode in general. Hence there exist left Hopf algebroids that are neither  full nor Hopf algebroids in the sense of \cite{JHLu}. 

\medskip

Motivated by \cite{KR1}, in the present note we give further examples of Hopf algebroids without an antipode, i.e., \emph{left} Hopf algebroids which are neither \emph{full} nor satisfy the axioms of \cite{JHLu}. In  these examples we consider Jacobi algebras, a generalization of Poisson algebras first introduced in a differential geometric context in \cite{kirillov,lichnerowicz}, and construct  {quotients} of  a canonical Lie--Rinehart algebra associated to them (see Lemma \ref{AA} in Section \ref{1jetA}). In certain cases, the universal enveloping algebra of some of these quotient  Lie--Rinehart algebras (see Section \ref{frak J})  will not admit an antipode. To prove this, we use a result by Kowalzig and Posthuma \cite{NielsPosthuma} which states that an antipode will exist on the universal enveloping algebra of a Lie--Rinehart algebra $(A,L)$ if and only if $A$ is a right $(A, L)$-module or equivalently there exist \textit{flat} right $(A,L)$-connections on $A$, see \cite{HuebschmannPaper1}, or {flat} right $(A, L)$-connections \textit{characters} on $A$, see Proposition \ref{proposition: connection characters}.

\medskip

More precisely, the aim of this note is twofold: 

\medskip

Firstly, in Section \ref{1jetA},  we focus on the algebraic characterization (as a Lie--Rinehart algebra) of a canonical Lie algebroid associated to Jacobi manifolds  proposed in \cite{ben, JacobiVaisman}. Our   description is equivalent to  the one given by Okassa  \cite{Okassa}, although we use a different approach: while  \cite{Okassa} uses so-called Jacobi 1- and 2-forms  on the trivial extension of a Jacobi algebra $A$ over a field $k$  by its  module of K\"ahler differentials $\Omega^1 (A)$, we  consider the $A$-module $A \oplus \Omega^1 (A)$ as a \textit{quotient} (see \cite[Chapter 10]{Matsumura} for this construction) which we call the 1-jet space of $A$ and denote by $\mathcal J^1 (A)$.  Furthermore, we  prove that  there exist  flat right $(A, \mathcal J^1 (A) )$-connections on $A$, and hence that the universal enveloping algebra of  the  Lie--Rinehart algebra $(A, \mathcal J^1 (A))$  admits an antipode, a fact which to our knowledge has not been stated  before in the literature.   We formulate these results in our first theorem.

\medskip

Throughout this note we fix a field $k$ and  denote tensor products of $k$-vector spaces with an unadorned $\otimes$.

\begin{theorem}
\label{Jacobi flat}
  Let $(A, \{ \bullet , \bullet \}_J )$ be a Jacobi algebra over a field $k$, $I$ be the kernel of the multiplication map $\mu: A \otimes A \rightarrow A$, $\mathcal J^1 (A) :=  ( A \otimes  A ) / I^2  \cong  A \oplus  \Omega^1 (A) $ be the 1-jet space of $A$ and      $j^1 : A \rightarrow \mathcal J^1 (A)$ be the 1-jet map $  a \mapsto 1 \otimes  a \pmod{ I^2}$, for all $ a \in A$.

\begin{enumerate}

\item The pair    $(A, \mathcal J^1 (A) )$  is a  Lie--Rinehart algebra with  anchor 
\begin{equation}
\label{anchor J}
\rho_{\mathcal  J^1  }: \mathcal  J^1 (A) \longrightarrow \mathrm{Der}_k (A), \quad j^1 (a) \longmapsto \Phi_a := \{ a , \bullet \}_J + \{ 1 , a \}_J \cdot \bullet
\end{equation}
and Lie bracket on $\mathcal J^1 (A)$ given by 
\begin{equation}
\label{equation: Bracket J}
\lbrack  j^1 (f),   j^1 (g) \rbrack_{ \mathcal J^1 (A) }  =   j^1 ( \{ f , g \}_J ) .
\end{equation}
\item    The map 
\begin{equation}
\label{connection}
\varphi_{\mathcal  J^1  }: \mathcal J^1 (A) \longrightarrow A , \quad a \cdot j^1 (b) \longmapsto   \{ a , b \}_J 
\end{equation}
is a flat right $(A, \mathcal J^1 (A) )$-connection character on $A$ which induces a flat right $(A, \mathcal J^1 (A) )$-connection on $A$ .
\end{enumerate}
Consequently, the universal enveloping algebra of the Lie--Rinehart algebra $(A, \mathcal J^1 (A) )$ admits an antipode.
\end{theorem}

Secondly, in Section \ref{section: ingredients}, we  give a method to   construct new Lie--Rinehart algebras associated to certain Jacobi algebras as a quotient of the canonical lift of $( A , \mathcal J^1 (A))$ to $(A , A \otimes  A)$. These new Lie--Rinehart algebras  do not admit antipodes in certain cases.
We summarise these results in our second theorem:

\begin{theorem}
\label{main}
Let $(A,  \{ \bullet , \bullet \}_J)$ be a Jacobi algebra over a field $k$,  let $h \in A$,  and assume that $ r \cdot \{ \bullet  , \bullet \}_J = 0$ for all $r \in \mathrm{Ann}_A( \{ h \})$. 
\begin{enumerate}
 
\item  The pair  $( A , A h \otimes  A  )$ is a Lie--Rinehart algebra with   anchor    $\rho_{ A h \otimes  A}: A h \otimes  A \rightarrow \mathrm{Der}_k (A)$  given by $ h \otimes  a \mapsto \Phi_a  $,  and  Lie bracket on $A h \otimes  A$   given by 
\begin{equation}
\label{LL}
\lbrack h \otimes  f , h \otimes  g \rbrack_{ A h \otimes  A }   = h \otimes  \{ f , g \}_J   . 
\end{equation}
\item Assume there exists a right $(A, Ah \otimes A)$-connection $\nabla^r$ on $A$.  
\begin{enumerate}

\item Then there  exists some  $a \in A$ satisfying $  a \cdot  r = \{ 1 , r \}_J$ for all $r \in \textnormal{Ann}_A( \{ h \} )$.

\item Moreover, if $\nabla^r$ is flat,  so that  $A$ is a right $(A , Ah \otimes A)$-module extending multiplication in $A$, then  for   all $b\in A$ satisfying $\{ 1 , b \}_J =0$, there exists an element $a \in A$ satisfying $a \cdot r = \{ 1 , r \}_J$ for all $r \in \textnormal{Ann}_A (\{h \} )$ and such that the following compatibility condition holds: 
 \begin{equation*}
 \{ b , a  \}_J =  \{ 1 ,  c \}_J, \quad  \textnormal{for some $c \in A$. }
 \end{equation*}
\end{enumerate}
\end{enumerate}
\end{theorem}
We see that if a Jacobi algebra $(A, \{ \bullet , \bullet \}_J)$ satisfying $r \cdot \{ \bullet , \bullet \}_J = 0$ for all $r  \in \mathrm{Ann}_A( \{ h \})$ for some fixed $h \in A$ does not satisfy conditions (a) or (b) in Theorem \ref{main} Part 2, then there is an obstruction to the existence of flat right $(A, Ah \otimes  A)$-connections on  $A$. By a result  of Kowalzig and Posthuma \cite{NielsPosthuma}, see Section \ref{connections} below,  the  universal enveloping algebra of the  Lie--Rinehart algebra $(A, Ah \otimes  A)$ associated to these Jacobi algebras  will  provide new examples of left Hopf algebroids without antipode.   Section \ref{main proof} is dedicated to   examples of this construction. \\

\noindent \textbf{Acknowledgements.} 
 The author also wishes to thank  Uli Kr\"ahmer, her PhD advisor, for helpful discussions  while working on  this paper; Stuart White for  advice about research; and  Gabriella B\"ohm for her encouragement and  hospitality.

\section{Background}
\label{section: background}

In this section we  recall the definition of a Lie--Rinehart algebra, see \cite{Herz, HuebschmannPCQ,HuebschmannTerm,RinehartForms} and of its universal enveloping algebra; we review the main tools used in this note, namely right $(A,L)$-connections and  right connection characters on $A$, see \cite{HuebschmannPaper1,NielsThesis,NielsPosthuma}; and give some background on Jacobi algebras, see \cite{kirillov,lichnerowicz}. 

\subsection{Lie--Rinehart algebras}
\label{LR}
The term \emph{Lie--Rinehart algebra} was coined by Huebschmann \cite{HuebschmannTerm}. However, this  algebraic structure, which was introduced by Herz \cite{Herz} under the name Lie pseudo-algebra (also known as \emph{Lie algebroid} \cite{Pradines} in a differential geometric context), had been developed and studied before as a  generalization of Lie algebras.  
 See \cite[Section 1]{HuebschmannPCQ} for some historical remarks on this development.

\begin{definition}
Let $R$ be a commutative ring with identity, $\left(A, \cdot \right)$ a commutative $R$-algebra and $\left( L , [\bullet, \bullet]_L \right) $ a Lie algebra over $R$.  A pair $\left(A, L \right)$ is called a \textbf{Lie--Rinehart algebra} over $R$ if $L$ has a left $A$-module structure  $ A \otimes_R L \rightarrow L$, $a \otimes_R \xi \mapsto a \cdot \xi$ for $a \in A, \xi \in L$,   and there is an  $A$-linear Lie algebra homomorphism $ \rho: L \rightarrow \mathrm{Der}_R \left( A \right)$, called the \textbf{anchor map}, satisfying the Leibniz rule
\begin{equation}
\label{LRLR}
[ \xi, a \cdot \zeta ]_L = a \cdot [ \xi , \zeta ]_L + \rho ( \xi ) ( a ) \cdot \zeta, \quad a \in A, \xi, \zeta \in L.
\end{equation}
\end{definition}

In what follows, $R$ will always be a commutative ring with identity and, unless stated otherwise, all algebras will be over $R$. 

\medskip

A fundamental example is $(A , \mathrm{Der}_R (A))$  where $A$ is a commutative algebra, and we consider the  usual Lie bracket and $A$-module structure on $\mathrm{Der}_R (A)$ with   anchor given by the identity. See references above for further examples and applications. 

\medskip

An important milestone in the development of these algebraic structures is  Rinehart's work  \cite{RinehartForms} in which he gives the structure of their universal enveloping algebra  (see \cite[Section 2]{RinehartForms}), generalizing the construction of the universal  enveloping algebra of a Lie algebra.

\begin{definition}[Rinehart \cite{RinehartForms}]  Let $(A,L)$ be a Lie--Rinehart algebra. Its \textbf{universal enveloping algebra},  denoted by $V(A,L)$, is the universal associative $R$-algebra with

 \begin{enumerate}

\item an $R$-algebra map $A \rightarrow V(A,L)$,
\item  a Lie algebra map $ \iota_L: (L, [ \bullet , \bullet ]_L) \rightarrow ( V(A,L), [ \bullet , \bullet ])$ given by $\xi \mapsto \iota_L (\xi ) = : \bar \xi$, where $[ \bullet , \bullet ]$ denotes the commutator in $V(A,L)$
\end{enumerate}
such that for all  $a \in A$ and $\xi \in L$ we have 
\begin{equation*}
[ \bar \xi , a ] = \rho ( \xi)(a), \quad a \bar \xi = \overline{a \cdot \xi}
\end{equation*}
 where the product in $V(A,L)$ is denoted by concatenation.
\end{definition}

\subsection{Right $(A,L)$-module structures and connections on $A$}
\label{connections}

The concepts of left (respectively, right) $(A,L)$-connections  and $(A,L)$-module structures on  $A$-modules were  introduced in \cite{HuebschmannPaper1}. There is an equivalence of categories between these module structures and left (respectively, right) $V(A,L)$-module structures.  While the anchor map defines a canonical left $V(A,L)$-module structure on $A$ itself \cite[Remark 3.10]{NielsPosthuma}, there is no  canonical  right $V(A,L)$-module structure on $A$, see \cite[Section 3.2.2]{NielsPosthuma} for discussion. In fact, as proved in \cite{KR1},  $A$  will not carry a   right $V(A,L)$-module structure in general.   We will only be considering right $(A,L)$-module structures (and connections) on $A$.

\begin{definition}
\label{definition: right connection}
Let $(A,L)$ be a Lie--Rinehart algebra.   A \textbf{ right $(A,L)$-connection on  $A$}, where $A$ is considered as a module  over itself,  is an $R$-linear  map   
$   \nabla^r:   A \otimes_R L \rightarrow  A$ given by $      a \otimes_{R} \xi \mapsto  \nabla^r (a \otimes_R \xi ) =: \nabla^r_\xi (a)$ 
satisfying
\begin{align}
   \label{equation: huebschmann1}
       \nabla^r_\xi (a \cdot b)  & =      \nabla^r_{a \cdot \xi } (b) \\
 &  = a \cdot  \nabla^r_\xi (b) -  \rho ( \xi ) (a) \cdot b , \quad a , b \in A , \xi \in L. \label{equation: huebschmann2}
 \end{align}
A   right $(A, L)$-connection on  $A$ is \textbf{flat} if the map $ \nabla^r: A \otimes_R L \rightarrow A$ turns $A$ into a right $L$-module, that is, $   \left(  [ \nabla^r_\xi, \nabla^r_\zeta ]_{\mathrm{Der}_R (A)} + \nabla^r_{[\xi, \zeta ]_L } \right) (a) = 0$ for $ a  \in A$, $\xi, \zeta \in L$.
\end{definition}

Since  right $(A, L)$-connections are not necessarily flat,  it is useful to define the following:

\begin{definition}
\label{definition: curvature operators}
Given a right $(A,L)$-connection on $A$, we define the operator $\Cr$ which we call \textbf{right $(A,L)$-curvature operator}, as follows:
\begin{equation*}
\Cr  :  L \otimes_R  L \otimes_R A \longrightarrow A;   \quad  \  ( \xi , \zeta , a) \longmapsto \Cr( \xi , \zeta ) (a) :=  \left(  [ \nabla^r_\xi, \nabla^r_\zeta ]_{\mathrm{Der}_R (A)} + \nabla^r_{[\xi, \zeta ]_L } \right) (a)
\end{equation*}
for $ a \in A, \xi, \zeta \in L$. We call $\Cr ( \xi , \zeta ) (a) \in A$ the \textbf{curvature} of the right $(A,L)$-connection $\nabla^r : A \otimes_R L \rightarrow A$ on $\xi , \zeta \in L$ evaluated at $a \in A$. 
\end{definition}

We are now ready to define the main object we will consider:

\begin{definition}
A flat right $(A,L)$-connection on $A$, that is, a connection with   $\Cr ( \xi , \zeta ) (a) = 0$ for all $a \in A$ and $\xi , \zeta \in L$, turns $A$ into a  \textbf{right $(A,L)$-module}. 

\end{definition}  
 Now, from \eqref{equation: huebschmann1} it follows that $\nabla^r ( a \otimes_R \xi ) = \nabla^r ( 1 \otimes_R a \cdot \xi )$. So we deduce that  a right $(A,L)$-connection on $A$  is in fact a certain map from $ A \otimes_A L $ to $A$, that is, a map from $L$ to $A$, satisfying a  Leibniz-type rule. More precisely, we have: 

\begin{proposition}
\label{proposition: connection characters}
There exists  a one-to-one correspondence between    right  $(A,L)$-connections on $A$ and operators
\begin{equation}
\dr  : L \longrightarrow A, \quad \xi \longmapsto \nabla^r_\xi (1_A) =: \dr ( \xi ), \quad \xi \in L \label{equation: right connection character}
\end{equation}
satisfying
\begin{equation}
 \label{equation: right connection character axi} 
              \delta^\nabla_r (a \cdot \xi )  =  a \cdot \delta^\nabla_r ( \xi ) - \rho ( \xi ) (a).
               \end{equation} 
The operator $\dr$ is called \textbf{right connection character.}
\end{proposition}

\begin{proof} 
The property of the map $\dr: L \rightarrow A$ given in   \eqref{equation: right connection character axi} follows from   \eqref{equation: right connection character} and \eqref{equation: huebschmann2} since  $ \delta^\nabla_r ( a \cdot \xi )   =  \nabla^r_{a \cdot \xi } (1_A)  = a \cdot \nabla^r_{\xi } (1_A) - \rho ( \xi ) (a) \cdot 1_A =  a \cdot \delta^\nabla_r ( \xi ) - \rho( \xi ) (a) $.   
\end{proof}
Note that this correspondence is   implicit in \cite[Theorem 1]{HuebschmannPaper1}, see also \cite[Theorem 4.2.7 and p85]{NielsThesis}.  
\medskip

Since a right $(A,L)$-connection $\nabla^r $ can be described in terms of a corresponding map $\dr$, a natural description of the operator $\Cr$ follows:

\begin{lemma}
\label{theorem: A-linear} 
In terms of the operator  $\dr$, the map  $\Cr : L \otimes_R L \otimes_R A \rightarrow A$ is  
      %
\begin{equation}
\Cr: L \otimes_R L \otimes_R A \longrightarrow A , \quad \left(   \xi , \zeta , c \right) \longmapsto    c \cdot \left( - \rho( \xi) \left( \dr  ( \zeta ) \right) + \rho \left( \zeta \right)  \left( \dr ( \xi ) \right) + \dr  \left(  [\xi, \zeta ]_L \right) \right) \label{equation: C_r 1}
\end{equation}
%
 where $ c \in A$ and  $\xi, \zeta \in L$. Moreover, the operator $\Cr$ is trilinear.
\end{lemma}

\begin{proof}
Using  Definition \ref{definition: right connection}  and Proposition \ref{proposition: connection characters} we  argue as follows
\begin{align*}
& \Cr   (a \cdot \xi , b \cdot \zeta ) (c)  = \nabla^r_{ a \cdot \xi } \left(       \nabla^r_{ b \cdot \zeta } (c) \right) -  \nabla^r_{ b \cdot \zeta } \left(       \nabla^r_{ a \cdot \xi } (c) \right) + \nabla^r_{  [ a \cdot \xi, b \cdot \zeta ]_L} (c) \\
& = \nabla^r_{ a \cdot \xi } \left( \dr (b \cdot c \cdot \zeta ) \right) - \nabla^r_{ a \cdot \xi } \left( \dr (a \cdot c \cdot \xi ) \right) + \dr \left( c \cdot [ a \cdot \xi , b \cdot \zeta ]_L \right) \\
 & = \dr \left( a \cdot  \dr (b \cdot c \cdot \zeta ) \cdot \xi \right) - \dr \left( b \cdot  \dr ( a \cdot c \cdot \xi ) \cdot \zeta \right)  \\
 & \phantom{={}}  + \dr \left( a \cdot b \cdot c \cdot [ \xi , \zeta ]_L + a \cdot c \cdot \rho ( \xi ) (b) \cdot \zeta - c \cdot b \cdot \rho ( \zeta ) (a) \cdot \xi  \right) \\
 & =  a \cdot b \cdot c \cdot \left( - \rho( \xi) \left( \dr  ( \zeta ) \right) + \rho ( \zeta )  \left( \dr ( \xi ) \right) + \dr  \left(  [\xi, \zeta ]_L \right) \right) . \qedhere
\end{align*}
\end{proof}
Since the operator $\Cr$ is $A$-linear, and in fact $A$-trilinear, we only  need to consider the curvature of a right $(A,L)$-connection at $1_A \in A$  and so, we  write $\Cr ( \xi , \zeta ) ( 1_A ) =: \Cr ( \xi , \zeta )$ for $ \xi , \zeta \in L$. Furthermore, since right $(A, L)$-connections on $A$ and right $(A,L)$-connnection characters on $A$ are equivalent, we will refer to the operator $\dr$ in \eqref{equation: right connection character} as \textit{connection}.

\medskip

In \cite[Proposition 3.11]{NielsPosthuma}, it is proved that there exists   an antipode on the universal enveloping algebra of $(A,L)$ turning the left Hopf algebroid structure  on $V(A,L)$ into a {full Hopf algebroid}  if and only if there exists a right $V(A,L)$-module structure on $A$.  From \cite{KR1}, it follows that the left Hopf algebroid  $V(A,L)$ is not, in general, a  full one. However, examples where $V(A,L)$ admits an antipode do exist: e.g. the universal enveloping algebra of a Lie algebroid \cite{WeinsteinTransverse} and of the  canonical Lie--Rinehart algebra associated to  Poisson algebras \cite[Section (3.2)]{HuebschmannPaper1}.

\subsection{Jacobi algebras}
\label{Jacobi algebras}

Jacobi algebras were first introduced by Kirillov \cite{kirillov} under the name ``local Lie algebras"  and independently by Lichnerowicz \cite{lichnerowicz} as the algebraic structure on the ring of $C^\infty$-functions on a certain kind of smooth manifolds, called Jacobi manifolds, see Section \ref{Jacobi section} below.  (See \cite[Section 2.2]{JacobiMarle} for some remarks comparing both definitions.)
Here we give a purely algebraic definition, see \cite{GRJacobi} for a graded version and \cite{Militaru} for results on Frobenius Jacobi algebras, representations of Jacobi algebras, and classification. 
\begin{definition}
A \textbf{Jacobi algebra} over $R$ is a commutative $R$-algebra $(A, \cdot )$ endowed with an $R$-linear  Lie bracket $\{ \bullet , \bullet \}_J$, called the Jacobi bracket, satisfying the  Leibniz rule
\begin{equation}
\label{equation: generalized leibniz}
\{a \cdot b, c\}_J = a \cdot \{ b, c \}_J + b \cdot \{a, c \}_J -  a \cdot b \cdot \{1, c \}_J, \quad a, b, c \in A.
\end{equation}

\end{definition}

Poisson algebras can be seen as Jacobi algebras where $\{ 1 , a \}_J = 0$ for all $a \in A$.

\begin{proposition} 
The  Jacobi bracket $\{ \bullet , \bullet  \}_J$ induces a Lie algebra map defined by 
\begin{equation}
\label{equation: Hamiltonian Jacobi}
\Phi : A \longrightarrow \mathrm{Der}_R (A), \quad a \longmapsto \Phi_a :=   \{a, \bullet \}_J + \bullet \cdot \{1, a\}_J    , \quad a \in A.
\end{equation}
\end{proposition}

\begin{proof}
First we check that $\Phi_a $ is a derivation on $A$  for all $a \in A$. We have
\begin{equation*}
\begin{split}
\Phi_a \left( b \cdot c \right) & = \{ a , b \cdot c \}_J + b \cdot c \cdot \{ 1 , a  \}_J = b \cdot \{ a , c \}_J + c \cdot \{ a , b \}_J + 2 \cdot b \cdot c \cdot \{1 , a \}_J \\
 & = b \cdot \Phi_a \left( c \right) + c \cdot \Phi_a \left( b \right) .
\end{split}
\end{equation*}
Furthermore,  using the fact that  $\{ \bullet , \bullet \}_J$ satisfies the Jacobi identity, a straightforward computation shows that the derivation $ \Phi_{ \{a, b \}_J } - \left[  \Phi_a, \Phi_b  \right]_{\mathrm{Der}_R(A)} $ vanishes on all $c \in A$. Hence $\Phi$ is a Lie algebra homomorphism.   \end{proof}

\begin{corollary}
The operator $\{ a , \bullet \}_J$ is a first order differential operator for all $a \in A$, and a derivation for the case $a = 1$. 
\end{corollary}
See also \cite[Proof of Lemma 4]{kirillov}.

\section{The Lie--Rinehart algebra $ (  A , \mathcal J^1 (A) )$ over a Jacobi algebra $A$}
\label{section: Jacobi flat}
 In this  section we develop the proof of Theorem \ref{Jacobi flat}.

\subsection{The 1-jet space of $A$}
\label{1jetA}

In this section we first describe the  $A$-module structure of the 1-jet space of  a commutative algebra $A$. Then we consider $A$ to be a Jacobi algebra and show that the pair $(A , A \otimes A)$ admits a Lie--Rinehart algebra structure (see Lemma \ref{AA}) which descends to a Lie--Rinehart structure on   $(A, \mathcal J^1(A))$, (Theorem \ref{Jacobi flat} Part 1).

\begin{definition}
Let $A$ be a commutative algebra, $\mu: A \otimes_R A \rightarrow A $ be the multiplication map $a \otimes_R b \mapsto a \cdot b$, and let $I = \mathrm{Ker}{\mu} \in A \otimes_R A$.
The \textbf{1-jet space of $A$} is an  $A$-module defined by 
\begin{equation}
\label{defJ}
\mathcal J^1 (A) := ( A \otimes_R A ) / I^2.
\end{equation} 
\end{definition}

We now provide a characterization of $\mathcal J^1 (A)$ as the trivial extension of $A$ by $\Omega^1 (A)$, the $A$-module of K\"ahler differentials over $A$, explaining its relation to \eqref{defJ}, see \cite[Chapter 10]{Matsumura} for more details.

\begin{proposition}
  There exists a canonical isomorphism of $A$-modules
\begin{equation}
\label{canonical iso}
( A \otimes_R A ) / I^2 \cong A \oplus \Omega^1 (A)  
\end{equation}
which identifies $a \otimes_R b \pmod{I^2} $ with $(a \cdot b , a \cdot db)$ for all $a , b \in A$.
\end{proposition}

\begin{proof}
First note $\Omega^1 (A) = I/ I^2$ and $da = 1 \otimes_R a - a \otimes_R 1 \pmod{I^2}$, so that  $a \cdot db = a \otimes_R b - a \cdot b \otimes_R 1 \pmod{I^2}$. Let $\lambda : A \rightarrow A \otimes_R A$ be given by $a \mapsto a \otimes_R 1$. Then for all $a , b \in A$ we  write $a  \otimes_R b  \in A \otimes_R A$ as   $ a  \otimes_R b = a \cdot b  \otimes_R 1 + ( a \otimes_R b - a \cdot b \otimes_R 1 ) $  where $a \cdot b  \otimes_R 1 \in \lambda (A)$ and $a \otimes_R b - a \cdot b  \otimes_R 1   \in I$. Since  $ I \cap \lambda (A)  = 0  $,  we deduce $ A  \otimes_R A   = \lambda \left( A \right) \oplus I$ and  $   \lambda (A)   / I^2 =  \lambda (A)$. Since $\lambda$ is injective, we can identify $\lambda(A)$ with $A$, 
hence 
\begin{equation*}
( A \otimes_R A ) / I^2  =  \left( \lambda (A) \oplus I \right) / I^2  = \left(  \lambda (A) / I^2 \right) \oplus \left( I /I^2 \right)   \cong A \oplus \Omega^1 (A)    . \qedhere
\end{equation*}
\end{proof}

\begin{proposition}
As an  $A$-module, $\mathcal J^1 (A)$ is generated by the image of $(A , \cdot )$ under the map $j^1 : A \rightarrow \mathcal J^1 (A)$ given by $a \mapsto 1 \otimes_R a \pmod{I^2}$ and called the \textbf{1-jet map}. The elements $j^1 (a) \in \mathcal J^1 (A)$ satisfy the Leibniz rule 
\begin{equation}
\label{leibniz J1}
 j^1  (a \cdot b ) - a \cdot j^1 (b) - b \cdot j^1 (a) + a \cdot b \cdot j^1 (1) = 0.
 \end{equation}
\end{proposition}

\begin{proof}
   For an element $\sum a_i  \otimes_R b_i \in A \otimes_R A$ we have $ \sum a_i \otimes_R b_i = \sum a_i  \cdot \left( 1 \otimes_R b_i \right) $. Since  we have $1 \otimes_R b_i \pmod{I^2} = j^1 \left( b_i \right)$ we deduce that any element in $\mathcal J^1 (A) = \left( A \otimes_R A \right) / I^2$ is of  the form $\sum a_i \cdot j^1 \left( b_i \right)$. Hence we deduce  that, as an $A$-module, $\mathcal J^1 (A)$ is generated  by   $j^1 (a) \in \mathcal J^1 (A)$ for all $a \in A$.  

As  $
1 \otimes_R a \cdot b - a \otimes_R b - b \otimes_R a + a \cdot b \otimes_R 1  =  \left( 1 \otimes_R a - a \otimes_R 1 \right) \cdot \left( 1 \otimes_R b - b \otimes_R 1 \right) \in I^2 $,  the Leibniz rule in \eqref{leibniz J1} holds. 
\end{proof}
Note that  the isomorphism $(A \otimes_R A) / I^2 \cong A \oplus \Omega^1 (A)$ given in \eqref{canonical iso}  identifies  $a \cdot j^1 (b) \in \mathcal J^1 (A) $ with $( a \cdot b , a \cdot db ) \in A \oplus \Omega^1 (A)$ for all $a , b \in A$.

\medskip

Our next aim is to endow the $A$-module  $A \otimes  A$, where $A$ is a Jacobi algebra over a field $k$, with a Lie bracket, denoted $[ \bullet , \bullet ]_{A \otimes  A}$, and an $A$-linear Lie algebra map from $A \otimes  A$ to $\mathrm{Der}_k (A)$ compatible with the bracket $[ \bullet , \bullet ]_{A \otimes  A}$ so that    the pair $(A , A \otimes  A)$ is a canonical Lie--Rinehart algebra over $A$. 

\begin{lemma}
\label{AA}
Let $(A , \{ \bullet , \bullet \}_J)$ be a Jacobi algebra over a field $k$, the pair $(A, A \otimes  A)$ is a Lie--Rinehart algebra with anchor 
\begin{equation}\label{anchorAA}
\rho_{A \otimes  A}: A \otimes  A \longrightarrow \mathrm{Der}_k (A), \quad   a \otimes  b \longmapsto a \cdot \Phi_b
\end{equation}
 and Lie bracket on $A \otimes  A$ given by 
\begin{equation}
\label{bracketAA}
[a \otimes  f , b \otimes  g ]_{A \otimes  A} = a \cdot b \otimes  \{ f , g \}_J + a \cdot \Phi_f (b) \otimes  g - b \cdot \Phi_g (a) \otimes  f.
\end{equation}
\end{lemma}

\begin{proof}
The bracket in  \eqref{bracketAA} is skew-symmetric and by a brief computation we see  that it satisfies the Jacobi identity. We now check that the $A$-linear map in \eqref{anchorAA} is a Lie algebra map:
\begin{align*}
\rho_{A \otimes  A} \left( [ a \otimes  f , b \otimes  g ]_{ A \otimes  A } \right) & = \rho_{A \otimes  A}  \left( a \cdot b \otimes  \{ f , g \}_J + a \cdot \Phi_f (b) \otimes  g - b \cdot \Phi_g (a) \otimes  f  \right)  \\
 & = a \cdot b \cdot \Phi_{  \{ f , g \}_J } + a \cdot \Phi_f (b) \cdot \Phi_g - b \cdot \Phi_g (a) \cdot \Phi_f  \\
 & = a \cdot b \cdot [ \Phi_f , \Phi_g]_{\mathrm{Der}_k (A)}   + a \cdot \Phi_f (b) \cdot \Phi_g - b \cdot \Phi_g (a) \cdot \Phi_f   \\
 & =  [a \cdot \Phi_f , b \cdot \Phi_g ]_{\mathrm{Der}_k (A)} \\
 &   = [ \rho_{A \otimes  A} ( a \otimes  f ) , \rho_{A \otimes  A} ( b \otimes  g ) ]_{\mathrm{Der}_k (A)}.
\end{align*}
Lastly, we show that  the bracket in \eqref{bracketAA} is  compatible with the anchor in \eqref{anchorAA} since the Leibniz rule in \eqref{LRLR} is satisfied:
\begin{align*}
& [ a \otimes  f , b \cdot c \otimes  g ]_{A \otimes  A}  = a \cdot b \cdot c \otimes  \{ f , g \}_J + a \cdot \Phi_f ( b \cdot c ) \otimes  g - b \cdot c \cdot \Phi_g (a) \otimes  f \\
 & = a \cdot b \cdot c \otimes  \{ f , g \}_J + a \cdot b \cdot \Phi_f (c) \otimes  g + a \cdot c \cdot \Phi_f (b) \otimes  g - b \cdot c \cdot \Phi_g (a) \otimes  f \\
 & = b \cdot (a \cdot c \otimes  \{ f , g \}_J + a \cdot \Phi_f (c) \otimes  g - c \cdot \Phi_g(a) \otimes  f) + a \cdot c \cdot \Phi_f (b) \otimes  g \\
 & = b \cdot [a \otimes   f , c \otimes  g ]_{A \otimes  A} +  \rho_{A \otimes  A} ( a \otimes  f)  (b) \cdot ( c \otimes  g). \qedhere
\end{align*}
%
\end{proof}

We are now ready to prove that the pair $( A , \mathcal J^1 (A) )$ admits a Lie--Rinehart algebra structure:   the Lie bracket $ [ \bullet , \bullet ]_{ A \otimes  A}$ on $A \otimes  A$ in \eqref{bracketAA} and the anchor map given by $\rho_{A \otimes  A}: A \otimes A \rightarrow \mathrm{Der}_k (A)$ in \eqref{anchorAA} descend respectively to a Lie bracket on $\mathcal J^1 (A)$ and to an $A$-linear Lie algebra map from $\mathcal J^1 (A)$ to $\mathrm{Der}_k (A)$  which are compatible with each other since they satisfy the Leibniz rule   \eqref{LRLR}.

\begin{proof}[Proof of Theorem \ref{Jacobi flat} Part 1]  
The Lie bracket $[ \bullet , \bullet ]_{ A \otimes  A}$ on $A \otimes  A$ given in \eqref{bracketAA} 
descends to the bracket $[ \bullet , \bullet ]_{ \mathcal J^1 (A) }$ given in \eqref{equation: Bracket J} if and only if it maps the $A$-module $A \otimes  A \otimes I^2 + I^2 \otimes  A \otimes  A$ to $I^2$. Since $[ \bullet , \bullet ]_{ A \otimes   A}$ is  a skew-symmetric bracket, it is enough to check that it maps $A \otimes  A \otimes  I^2$ to $I^2$.    Let  $a_1 \otimes  b_1 \in A \otimes  A  $ and $ \sum  a_i \otimes  b_i , \sum f_j \otimes  g_j  \in I $ so that $\sum a_i \cdot b_i = \sum f_j \cdot g_j = 0$.    
Then we have $ (a_1 \otimes   b_1 )  \otimes  \left(  \left( \sum a_i  \otimes  b_i \right) \cdot \left( \sum f_j  \otimes   g_j   \right)  \right)  \in  A \otimes  A  \otimes   I^2    $   which under  $ [ \bullet , \bullet ]_{ A \otimes  A}$
becomes
%
%
        \begin{align*}
 & \left[ a_1 \otimes  b_1   , \left( \sum a_i \otimes  b_i \right) \cdot \left( \sum f_j \otimes  g_j \right)  \right]_{ A \otimes  A} =   \left[  a_1  \otimes  b_1 , \sum a_i \cdot f_j \otimes  b_i \cdot g_j   \right]_{ A \otimes  A} \\
 & =   a_1 \cdot  \left( \sum  a_i \otimes  b_i \right)  \cdot \left(  \sum  f_j \otimes   \{ b_1 , g_j \}_J \right) + a_1 \cdot \left( \sum f_j \otimes   g_j \right) \cdot \left( \sum  a_i \otimes  \{ b_1 , b_j \}_J \right)   \\
 & \phantom{={}} + a_1 \cdot  \left( \sum a_i \otimes  b_i \right) \cdot \left( \sum  \{ b_1 , f_j \}_J \otimes  g_j \right) + a_1 \cdot \left( \sum f_j \otimes  g_j \right) \cdot \left(  \sum \{ b_1 , a_i \}_J \otimes  b_i \right) \\
& \phantom{={}} + a_1 \cdot \left( \sum a_i \otimes b_2 \right) \cdot \left( \sum f_j \otimes g_j \right) \cdot ( 1 \otimes \{ 1 , b_1 \}_J + 2 \{ 1 , b_1 \}_J \otimes 1 )  \in I^2 
    \end{align*}
since we have $ \sum  a_i \otimes  b_i \in I$, $ \sum f_j \otimes  g_j \in I$,  $ \sum(  f_j \otimes  \{ b_1 , g_j \}_J + \{ b_1 , f_j \}\otimes  g_j ) \in I$ and  $ \sum ( a_i \otimes  \{ b_1 , b_i \}_J + \{ b_1 , a_i \}_J \otimes  b_i ) \in I$.
Hence  $ A \otimes  A \otimes   I^2  +  I^2 \otimes  A \otimes  A $  is mapped to $ I^2$ under $[ \bullet , \bullet ]_{ A \otimes  A}$  and consequently  $[ \bullet , \bullet ]_{ A \otimes  A}$      descends to the bracket  $[ \bullet , \bullet ]_{ \mathcal  J^1 (A) }$  in   \eqref{equation: Bracket J}. 
A  straightforward computation shows that  $[ \bullet , \bullet ]_{ \mathcal  J^1 (A) }$ also satisfies the Jacobi identity, hence it is a Lie bracket.

\medskip

By a similar argument, we show that the map  $\rho_{ A  \otimes  A}$ in \eqref{anchorAA}  descends to the map   in \eqref{anchor J}.  Let $a \otimes  b, f \otimes  g \in I$ so that $a \cdot f \otimes  b \cdot g  \in I^2$, which under $\rho_{ A \otimes  A}$ becomes
\begin{equation*}
\rho_{ A \otimes  A} (a \cdot f \otimes  b \cdot g) = a \cdot f \cdot \Phi_{b \cdot g} = a \cdot f \cdot \{ b \cdot g , \bullet \}_J + a \cdot f \cdot \{ 1 , b \cdot g \}_J \cdot \bullet =  0.
\end{equation*} 
Since $\rho_{ A \otimes  A}$ maps $I^2$ to 0, we deduce that it descends to the map $\rho_{ \mathcal J^1}$.

\medskip

Moreover, a short computation similar to  one performed in the proof of Lemma \ref{AA} shows that the Leibniz rule in \eqref{LRLR} holds. 
Thus, the  Lie bracket $ [ \bullet , \bullet ]_{\mathcal J^1 (A)}$ in \eqref{equation: Bracket J}      and the anchor  $\rho_{\mathcal J^1}$  in \eqref{anchor J} turn $(A , \mathcal J^1 (A))$ into  a Lie--Rinehart algebra.
\end{proof}

\subsection{Relation of $(A , \mathcal J^1 (A))$ to the Lie algebroid over a Jacobi manifold}

\label{Jacobi section}

In this brief section we explain how our algebraic description, given in Theorem \ref{Jacobi flat}, of the canonical Lie--Rinehart algebra $(A, \mathcal J^1 (A))$ associated to a Jacobi algebra $(A , \{ \bullet , \bullet \}_J)$ is related to the geometric description of the Lie algebroid over a Jacobi manifold $M$ as given in \cite{ben}. 
\begin{definition}
Let $M$ be a smooth manifold equipped with a bivector field $\Lambda$ and a vector field $E$. The ring  $C^\infty (M)$ admits a Lie bracket, called a \textbf{Jacobi structure}, given by
\begin{equation}
\label{JKL}
\{ f, g \}_J = \Lambda \left( df , dg \right) + f \cdot E \left( g \right) - g \cdot E \left( f \right), \quad f , g \in C^\infty \left( M \right)
\end{equation} 
if and only if  $\Lambda$ and  $E$ satisfy    $ \llbracket \Lambda, \Lambda \rrbracket = 2 E \wedge \Lambda$ and $ \llbracket \Lambda, E \rrbracket = 0$,  where $\llbracket \bullet , \bullet \rrbracket$ is the Schouten bracket. Then $(M, \Lambda, E)$ is called a \textbf{Jacobi manifold}.

\end{definition}

In order to construct the canonical Lie algebroid associated to a Jacobi manifold $(M, \Lambda, E)$, Kerbrat and Souici-Benhammadi  \cite{ben} endow the bundle $J^1(M, \mathbb R)$ of 1-jets of smooth functions on the manifold $M$, which is isomorphic as a  $C^\infty (M)$-module to the direct sum $C^\infty (M) \oplus \Omega^1 (M)$ where $\Omega^1 (M)$ are the smooth differential 1-forms, with a Lie bracket given by $ [(f_1 , a_1 \cdot db_1) , ( f_2 , a_2 \cdot db_2)] = ( f , a \cdot d b) $
where 
\begin{align*}
f &  =  - \Lambda( a_1 \cdot db_1 , a_2 \cdot db_2) \\
 &   \phantom{={}} + \iota ( \Lambda^\# (a_1 \cdot db_1) + f_1 \cdot E) \cdot d f_2 - \iota (\Lambda^\# ( a_2 \cdot db_2) + f_2 \cdot E) \cdot df_1 \\
 a \cdot db & = \mathcal L ( \Lambda^\# (a_1 \cdot db_1) + f_1 \cdot  E) a_2 \cdot db_2 - \mathcal L ( \Lambda^\# ( a_2 \cdot db_2) + f_2 \cdot E) a_1 \cdot db_1 \\
 & \phantom{={}}  - \langle a_1 \cdot db_1 , E \rangle ( a_2 \cdot db_2 - df_2 )  + \langle a_2 \cdot db_2 , E \rangle ( a_1 \cdot db_1 - df_1)  \\
 &  \phantom{={}} - d (\Lambda(a_1 \cdot db_1 , a_2 \cdot db_2 )) 
\end{align*}
for $f, f_1, f_2 \in C^\infty(M)$ and $a \cdot db, a_1 \cdot db_1 , a_2 \cdot db_2 \in \Omega^1 (M)$, which we can write in terms of the Jacobi structure in \eqref{JKL} as
\begin{align*}
f & = - a_1 \cdot a_2 \cdot \{ b_1 , b_2 \}_J + a_1 \cdot a_2 \cdot b_1 \cdot \{ 1 , b_2 \}_J - a_1 \cdot a_2 \cdot b_2 \cdot \{ 1 , b_1 \}_J \\
 &  \phantom{={}}+ a_1 \cdot \{ b_1 , f_2 \}_J - a_1 \cdot b_1 \cdot \{ 1 , f_2 \}_J + a_1 \cdot f_2 \cdot \{ 1 , b_1 \}_J + f_1 \cdot \{ 1 , f_2 \}_J \\
 &  \phantom{={}} - a_2 \cdot \{ b_2 , f_1 \}_J  +  a_2 \cdot b_2 \cdot \{ 1 , f_1 \}_J - a_2 \cdot f_1 \cdot \{ 1 , b_2 \}_J - f_2 \cdot \{ 1 , f_1 \}_J   \\
a \cdot db & = a_1 \cdot a_2 \cdot d \{ b_1 , b_2 \}_J \\
 &  \phantom{={}} + ( a_2 \cdot \{ a_1 , b_2 \}_J - f_2 \cdot \{ 1 , a_1 \}_J + a_2 \cdot b_2 \cdot \{ 1 , a_1 \}_J - a_1 \cdot a_2 \cdot \{ 1 , b_2 \}_J ) \cdot db_1 \\
 &  \phantom{={}}  - ( a_1 \cdot \{ a_2 , b_1 \}_J - f_1 \cdot \{ 1 , a_2 \}_J + a_1 \cdot b_1 \cdot \{ 1 , a_2 \}_J - a_1 \cdot a_2 \cdot \{ 1 , b_1 \}_J ) \cdot db_2 \\
 &  \phantom{={}}  - ( a_1 \cdot f_2 - a_1 \cdot a_2 \cdot b_2 ) \cdot d \{ 1 , b_1 \}_J  + ( a_2 \cdot f_1 - a_1 \cdot a_2 \cdot b_1 ) \cdot d \{ 1 , b_2 \}_J .
\end{align*}
Now, take $A = C^\infty (M)$ in \eqref{equation: Bracket J}. By the isomorphism given in  \eqref{canonical iso} we can identify an element    $a \cdot j^1 (b) + (f - a \cdot b  ) \cdot j^1 (1) \in \mathcal J^1 (C^\infty(M))$ with  $( f , a \cdot db) \in C^\infty (M) \oplus \Omega^1 ( C^\infty(M) )$. By the universal property of   K\"ahler differentials, a straightforward computation shows that  the bracket (on the algebraic K\"ahler differentials) given in   \eqref{equation: Bracket J}  \emph{descends} to the bracket (on differential forms)  defined in \cite{ben}, so that both constructions are in fact compatible.

\medskip

Furthermore, since elements   $a \cdot j^1 (b) \in \mathcal J^1 ( C^\infty (M) )$ are identified, as above, with elements  $(a\cdot b , a \cdot db ) \in C^\infty (M) \oplus \Omega^1 ( M)$, we see that the anchor map we defined in \eqref{anchor J} yields the anchor defined  in \cite{ben}, that is  a map $ \rho : J^1 (M, \mathbb R) \rightarrow TM$ given by $( f, a \cdot db ) \mapsto \Lambda^\# ( a \cdot db) + f \cdot E$.

\subsection{The Jacobi algebra $(A, \{ \bullet , \bullet \}_J)$ as a right $(A, \mathcal J^1 (A))$-module }

We are now ready to prove the main result stated in Theorem \ref{Jacobi flat}, which we recall here:

\begin{theorem}
\label{thm 3.6}
Let $( A , \mathcal J^1 (A))$ be the canonical Lie--Rinehart algebra associated to a Jacobi algebra $( A , \{ \bullet , \bullet \}_J)$. 
 The map 
\begin{equation}
\label{connection2}
\varphi_{\mathcal  J^1  }: \mathcal J^1 (A) \longrightarrow A , \quad a \cdot j^1 (b) \longmapsto   \{ a , b \}_J 
\end{equation}
is a flat right $(A, \mathcal J^1 (A) )$-connection character on $A$.

\end{theorem}

\begin{proof}
We first check that the  map     \eqref{connection2}  is well-defined. The map $ \gamma : A \otimes  A \rightarrow A$ given by $a \otimes  b \mapsto     \{ a , b \}_J  $ induces the map  in   \eqref{connection2} if and only if $\gamma \left(  I^2 \right) = 0$. Let $ \sum a_i \otimes  b_i ,  \sum f_j \otimes  g_j \in I$, so that $\sum a_i \cdot b_i = \sum f_j \cdot g_j = 0$.   Then, we have $ \left( \sum a_i \otimes  b_i \right) \cdot  \left( \sum f_j \otimes  g_j \right) =  \sum a_i \cdot f_j \otimes  b_i \cdot g_j \in  I^2$  which under the map $\gamma : A \otimes  A  \rightarrow A $ becomes 
%
\begin{align*} 
\gamma & \left( \sum a_i \cdot f_j \otimes  b_i \cdot g_j \right) = \sum  \{ a_i \cdot f_j , b_i \cdot g_j \}_J  \\
 &   = \sum a_i \cdot g_j \cdot \{ f_j , b_i \}_J + \sum f_j \cdot b_i \cdot \{ a_i , g_j \}_J + \sum  a_i \cdot f_j \cdot \{g_j , b_i \}_J - \sum  a_i \cdot f_j \cdot \{ g_j , b_i \}_J \\
 & = \sum  a_i \cdot \{ f_j \cdot g_j , b_i \}_J + \sum f_j \cdot \{ a_i \cdot b_i , g_j \}_J =0
 \end{align*}
%
%
so $ \varphi_{ \mathcal J^1  } : \mathcal J^1 (A) \rightarrow A$ is well-defined.    We now prove the map $\varphi_{ \mathcal J^1  } : \mathcal J^1  (A) \rightarrow A$ is a right $\left( A , \mathcal J^1 (A)  \right)$-connection character on $A$. 
Let $\sum a_i \cdot j^1 \left( b_i \right)  \in \mathcal J^1 (A)$, then we have
\begin{align*} 
\varphi_{ \mathcal J^1  } \left( \sum c \cdot a_i \cdot j^1 (b_i ) \right)  & = \sum \{ c \cdot a_i, b_i \}_J  \\
 & = \sum ( c \cdot \{ a_i , b_i \}_J  - a_i \cdot \{ b_i , c  \}_J - c \cdot a_i \cdot \{ 1 , b_i \}_J ) \\
 & = c \cdot \varphi_{ \mathcal J^1} \left( \sum a_i \cdot j^1 (b_i) \right) - \rho_{ \mathcal J^1 } \left( \sum  a_i \cdot j^1 (b_i) \right) ( c ) 
\end{align*}
so $\varphi_{ \mathcal J^1}$ satisfies \eqref{equation: right connection character axi}. Lastly, using  the identity   \eqref{equation: C_r 1} in Lemma \ref{theorem: A-linear}, we  compute the curvature of the $(A, \mathcal J^1  (A) )$-connection (character) on $A$ given in \eqref{connection2}:
\begin{small}
\begin{align*}
\Cr  \left( j^1 (a) , j^1 (b) \right)   & = - \rho_{  \mathcal J^1   } \left( j^1 (a) \right) \left( \varphi_{  \mathcal J^1    } \left( j^1 (b) \right) \right)  +\rho_{ \mathcal J^1  }  \left( j^1 (b) \right)  \left( \varphi_{  \mathcal J^1    } \left( j^1 (a) \right) \right) + \varphi_ {\mathcal J^1 } \left( [ j^1 (a) , j^1  (b) ]_{ \mathcal J^1 (A) }     \right)  \\
       & = -  \Phi_a \left( \{ 1 , b \}_J  \right) + \Phi_b \left( \{ 1 , a \}_J  \right)  + \varphi_{   \mathcal J^1} \left(  j^1 \{ a , b \}_J    \right)   \\
 & = - \{ a , \{ 1 , b \}_J   \}_J - \{ 1 , a \}_J  \cdot \{ 1 , b \}_J    +  \{ b ,  \{ 1 , a \}_J   \}_J +  \{ 1 , b \}_J \cdot \{ 1 , a \}_J    + \{ 1 ,  \{ a , b \}_J \}_J \\
 &     = 0  
 \end{align*}
\end{small}
\vspace{-10pt}
\noindent for all $a , b \in A$.
\end{proof}

\section{Other quotient Lie--Rinehart algebras of $(A, A \otimes A)$  }
\label{section: ingredients}

In the previous section we considered $(A, \mathcal J^1 (A))$ as a quotient of the Lie--Rinehart algebra $(A, A \otimes  A)$  associated to a Jacobi algebra over a field $k$. In  this section we construct new quotient Lie--Rinehart algebras  of $(A , A \otimes  A)$.

\subsection{Quotient Lie--Rinehart algebras}

\begin{lemma}
\label{quotient LR}
Let $(A,L)$ be a Lie--Rinehart algebra with Lie bracket on $L$ denoted by $[ \bullet , \bullet ]_L$ and anchor $\rho_L$. For  $h \in A$,  define $ \mu_h : L \rightarrow L$  by $ \zeta \mapsto h \cdot \zeta$ for $\zeta \in L$ and put $M := \mathrm{Im} ( \mu_h) = h L $, $K = \mathrm{Ker} ( \mu_h) = \{ \xi \in L \mid h \cdot \xi = 0 \}$. The pair $(A, M)$ admits a  Lie--Rinehart algebra structure with Lie bracket on $M$ given by 
\begin{equation}
\label{bracket K}
[ h \cdot \zeta , h \cdot \gamma ]_M := h \cdot [ \zeta, \gamma  ]_L, \quad \zeta , \gamma \in L,
\end{equation}
and anchor 
\begin{equation}
\label{anchor K}
\rho_M: (M, [ \bullet , \bullet ]_M)  \longrightarrow \mathrm{Der}_k (A), \quad  \rho_M ( h \cdot \zeta ) := \rho_L ( \zeta), \quad   \zeta \in L
\end{equation}
turning the  map  
\begin{equation}
\mu_h : ( L , [ \bullet  , \bullet  ]_L ) \longrightarrow ( M , [ \bullet , \bullet ]_M ), \quad \zeta  \longmapsto h \cdot \zeta  
\end{equation}
into a  Lie--Rinehart algebra homomorphism   if and only if 
\begin{enumerate}
\item $ \rho_L( \xi )  =0$, for all $\xi \in K$, \label{assumption 1}
\item $K$ is a Lie ideal in $L$, i.e., $ h \cdot [\xi , \bullet]_L = 0$, for all $\xi \in K$.

\end{enumerate}

\end{lemma}

\begin{proof}
Assume first that $(A , L )$ and $h \in A$ satisfy conditions (1) and (2) above. 
The bracket $[ \bullet , \bullet ]_M : M \otimes M \rightarrow M$, defined by $[ \bullet , \bullet ]_M \circ ( \mu_h \otimes \mu_h ) = \mu_h \circ [ \bullet , \bullet ]_L$, is well-defined since we have  $K = \mathrm{Ker} ( \mu_h)$ and $[ K , M ]_L \subset K$ by condition (2).  Moreover, $[ \bullet , \bullet ]_M$ is skew-symmetric and satisfies the Jacobi identity since $[ \bullet , \bullet ]_L$ does. Furthermore, we have   $\rho_M ( \theta) = \rho_M ( h \cdot \xi' ) = \rho_L ( \xi')$  which  vanishes by our assumptions, so the map $\rho_M : M \rightarrow \mathrm{Der}_k (A)$ is well-defined.  We next check that  $\rho_M$ is a Lie algebra map:
\begin{equation*}
\begin{split}
\rho_M ( [ h \cdot \zeta , h \cdot \gamma ]_M )  & = \rho_M ( h \cdot [ \zeta , \gamma ]_L )  = \rho_L ( [ \zeta , \gamma ]_L ) ) = [ \rho_L ( \zeta) , \rho_L ( \gamma) ]_L \\
 & =  [ \rho_M ( h \cdot \zeta) , \rho_M ( h \cdot \gamma)  ]_{\mathrm{Der}_k (A)}.
\end{split}
\end{equation*}
We finally  check that $[ \bullet , \bullet ]_M$ is compatible with $\rho_M$ since  the Leibniz rule in \eqref{LRLR} is satisfied:
\begin{equation*}
\begin{split}
[ h \cdot \zeta , a \cdot h \cdot \gamma  ]_M  & = h \cdot [ \zeta , a \cdot \gamma ]_L = h \cdot a \cdot [ \zeta , \gamma ]_L + h \cdot \rho_L ( \zeta ) ( a ) \cdot \gamma \\
 & = a \cdot [ h \cdot \zeta ,  h \cdot \gamma ]_M + \rho_M ( h \cdot \zeta) ( a ) \cdot (h \cdot \gamma) .
\end{split}
\end{equation*}
Conversely,  assume    $M$ admits the Lie--Rinehart algebra structure with bracket on $M$ given by \eqref{bracket K} and anchor given by \eqref{anchor K}. 
 Since for all $\xi \in K$ we have $ \rho_M ( h \cdot \xi ) = 0$, by \eqref{anchor K} we deduce $\rho_L ( \xi ) = 0$.
 Moreover, for elements $\xi \in K$ and $\gamma \in L$ we have $ [ h \cdot \xi , h \cdot \gamma   ]_M = 0 = h \cdot [ \xi , \gamma ]_L$ so $[ \xi , \gamma ]_L \in K$ for all $\gamma \in L$, so that  $K$ is a Lie ideal in $L$.  
\end{proof}

Note that since  $[ h \cdot \xi , \gamma_1 ]_L = 0$ for all $\xi \in K$,  by the Leibniz rule in  \eqref{LRLR} we deduce that $ \rho_L ( \gamma_1  ) (h ) \cdot \xi = 0$ for all $\gamma_1 \in L$ so that $  [ \rho_L ( \gamma_1 ) (h ) \cdot \xi , \gamma_2 ]_L = 0$ and $\rho_L ( \gamma_2 ) \circ \rho_L ( \gamma_1) (h) \cdot \xi = 0$. Repeating this iteration process we deduce   $ \rho_L ( \gamma_i ) \circ \cdots \circ \rho_L ( \gamma_1 ) ( h) \cdot \xi = 0$.

\medskip

We will base the proof of Theorem \ref{main} Part 2 (a) on the following observation:

\begin{lemma}
\label{fundamental}
 If there exist some $\zeta \in L$, $\xi \in K$  such that $\xi = a \cdot \zeta$ for some $a \in A$, and no $b \in A$ satisfying $a \cdot b = \rho_L ( \zeta ) (a)$, then there exists no right $(A,M)$-connection on $A$.

\end{lemma}

\begin{proof}
First note that since $\xi \in K$ we have $h \cdot \xi = a \cdot h \cdot \zeta = 0$. Now assume there exists a right $(A, M)$-connection character  $\delta: M \rightarrow A$ on $A$. Then we have   $$  0 = \delta ( a \cdot h \cdot \zeta ) = a \cdot \delta ( h \cdot \zeta) - \rho_M ( h \cdot \zeta ) ( a ) = a \cdot \delta ( h \cdot \zeta ) - \rho_L ( \zeta ) ( a )  $$ which is a contradiction. 
\end{proof}

\subsection{The $A$-module $A h \otimes  A$ for a Jacobi algebra $A$}
\label{frak J}

 Our aim in this section is to  endow the $A$-module $A h \otimes  A$, where $(A , \{ \bullet , \bullet \}_J)$ is a  Jacobi algebra over a field $k$ and $h  \in A$,   with a Lie bracket that is compatible with the  bracket on $\mathcal J^1(A)$ in \eqref{equation: Bracket J}.

\begin{lemma}
Let $(A , \{ \bullet , \bullet \}_J ) $  be  a Jacobi algebra over a field $k$, and let $h \in A$ be such that  $\{ \bullet , \bullet \}_J$ satisfies   $r  \cdot \{ \bullet  , \bullet \}_J = 0$  for all $r \in  \mathrm{Ann}_A( \{ h \} )      $. Then $( A, Ah \otimes  A )$ admits a Lie--Rinehart algebra structure with Lie  bracket on $Ah \otimes  A$ given by 
\begin{equation}
\label{bracket AhA}
 [ h \cdot a \otimes  f , h \cdot b \otimes  g ]_{Ah \otimes  A}: = h \cdot [ a \otimes  f , b \otimes  g ]_{ A \otimes  A}
\end{equation}
 and anchor
\begin{equation}
\label{anchor AhA}
\rho_{Ah \otimes  A} ( h \cdot a \otimes  f) =  \rho_{A \otimes  A} ( a \otimes  f ) = a \cdot \Phi_f.
\end{equation}
\end{lemma}

\begin{proof}
First recall that $(A , A \otimes  A)$ is a Lie--Rinehart algebra with Lie bracket  given in \eqref{bracketAA}   and anchor given in \eqref{anchorAA}.  Note also that there exists a map of $A$-modules given by  
\begin{equation*}
\mu_h  : A \otimes  A  \longrightarrow Ah \otimes  A , \quad 1 \otimes  a \longmapsto h \otimes  a.
\end{equation*}
We now prove that the conditions  in Lemma \ref{quotient LR} are satisfied. Let  $\{ e_i \}$ be a basis of the Jacobi algebra $(A, \{ \bullet , \bullet \}_J )$ as a $k$-vector space, and  let $K = \{ \sum a_i \otimes  e_i \in A \otimes  A \mid \sum a_i \cdot h \otimes   e_i = 0\}$. Now, let $ \sum a_i \otimes  e_i \in K$. Then $ a_i \cdot h = 0$ so $a_i \in \mathrm{Ann}_A ( \{ h \} )$ for all $i$.

\begin{enumerate}

\item Since $r \cdot \{ \bullet , \bullet \}_J = 0$ for all $r  \in \mathrm{Ann}_A ( \{ h \})$, we have $r  \cdot \Phi_{ \bullet } = 0$. Then, for elements $\sum a_i \otimes_R e_i \in K$ we have $ \rho_{A \otimes  A } ( \sum a_i \otimes  e_i ) = \sum a_i \cdot \Phi_{e_i} = 0$ since  $a_i \in \mathrm{Ann}_A ( \{ h \})$.

\item  
Since $\{ h \cdot r , \bullet \}_J = r \cdot \{ \bullet , \bullet \}_J = 0$,  the Leibniz rule       in 
\eqref{equation: generalized leibniz}  yields $ h \cdot \{ r , \bullet  \}_J =0 $  for all $r \in \mathrm{Ann}_A ( \{ h \})$. Furthermore, by computing the  Lie bracket on $A \otimes  A$  in \eqref{bracketAA} we obtain, for  $\sum a_i \otimes e_i \in K$,  $\sum b_j \otimes e_j \in A \otimes A$:
\begin{equation*}
\begin{split}
 \left[ \sum a_i \otimes  e_i , \sum b_j \otimes  e_j \right]_{A \otimes  A} & = \sum \sum a_i \cdot b_j \otimes  \{ e_i ,  e_j \}_J \\
 &  \phantom{={}} + \sum \sum a_i \cdot \Phi_{e_i} ( b_j) \otimes  e_j \\
 & \phantom{={}} - \sum \sum b_j \cdot \Phi_{e_j} ( a_i) \otimes  e_i
\end{split} 
\end{equation*}
 hence we deduce  $\left[ \sum a_i \otimes  e_i , \sum b_j \otimes  e_j \right]_{A \otimes  A} \in K$ so  that    $K$ is a Lie ideal in $(A \otimes  A , [ \bullet , \bullet ]_{A \otimes  A}$).

\end{enumerate}

Hence,  by Lemma \ref{quotient LR}, the pair  $(A, A h \otimes  A)$ can be endowed with a Lie--Rinehart algebra structure with anchor $\rho_{Ah \otimes   A} ( h \cdot a \otimes  f) =  \rho_{A \otimes  A} ( a \otimes  b ) = a \cdot \Phi_b$ and Lie bracket on $Ah \otimes  A$ given by  $ [ h \cdot a \otimes  f , h \cdot b \otimes  g ]_{Ah \otimes  A} = h \cdot [ a \otimes  f , b \otimes  g ]_{A \otimes  A}$. 
\end{proof}

The compatibility between  the Lie--Rinehart algebra structures in $(A , A \otimes   A)$, $(A , \mathcal J^1 (A))$ and $(A , A h \otimes  A)$ can be described using  the following commutative diagram: 
 \begin{displaymath}
 \xymatrix{
 & A \ar[dl]_{ \pmod{ I^2}} \otimes  A  \ar[dd]^{ \rho_{ A \otimes  A} } \ar[dr]^{ \mu_h } & \\
 \mathcal J^1 (A)  \ar[dr]_{\rho_{ \mathcal J^1   } }  &   & A h \otimes  A \ar[dl]^{\rho_{A h \otimes  A} }\\
 & \mathrm{Der}_k (A) &  }
 \end{displaymath}

\subsection{Right $(A, Ah \otimes  A)$-connections on $(A, \{ \bullet , \bullet \}_J )$. Proof of Theorem \ref{main}}

Throughout this section we assume $A$ is a Jacobi algebra over a field $k$, $h \in A$ and  $r \cdot \{ \bullet  , \bullet \}_J = 0$ for all $r \in \mathrm{Ann}_A( \{ h \})$.

\begin{lemma}
\label{twist}
A $k$-linear map $ \varphi_h : A h \otimes  A \rightarrow A$ is a  right $(A, A h \otimes  A )$-connection  on $A$ if and only if it is of the form
\begin{equation}
\label{j connection}
\varphi_h : A h \otimes  A \longrightarrow A, \quad a \cdot h \otimes  b \longmapsto a \cdot \dt  (b)  + \{  a , b \}_J  -  a \cdot \{ 1 , b \}_J
\end{equation}
where   $ \dt : A \rightarrow A$ satisfies:
\begin{equation}
\label{very important 2}
r  \cdot \dt (a) - \{ a , r \}_J = 0, \quad \forall r \in \mathrm{Ann}_A ( \{ h \}), \forall a \in A .
\end{equation}
In terms of  $\dt $, the curvature of a right $(A, Ah \otimes  A )$-connection on $A$ is  
\begin{equation}
\label{equation: curv second}
\Cr  \left( h \otimes  a , h \otimes  b  \right) =   \dt  \{  a , b \}_J  - \{ a , \dt (b)  \}_J  -   \{  \dt (a) , b  \}_J    - \{ 1 , a \}_J \cdot  \dt (b)   +  \dt  (a) \cdot \{  1 , b \}_J.
\end{equation}
\end{lemma}

\begin{proof}
We start by recalling that $A h \otimes  A$ is generated as an $A$-module by elements $h \otimes  a \in  Ah \otimes  A$,  $a \in A$. Let   $\varphi_h$ be a right $(A, A h \otimes  A )$-connection character on $A$, then it satisfies 
\begin{equation*}
\varphi_h : A h \otimes  A \longrightarrow A , \quad  a \cdot h \otimes  b   \longmapsto  a \cdot \varphi_h \left( h \otimes  b \right)   -  \rho_{ A h \otimes  A} \left( h \otimes  b \right) (a).
\end{equation*}
Let  $ \dt (a) := \varphi_h  \left( h \otimes  a \right), \forall a \in A$. Then we have
\begin{equation}
\label{well}
0 = \varphi_h  \left( r \cdot h \otimes  a \right) = r \cdot \dt (a) - \{ a , r \}_J - \{ 1 , a \}_J \cdot r =  r  \cdot \dt (a) - \{ a , r \}_J  .
\end{equation}

We now prove the converse statement:
\medskip

A $k$-linear  map $\varphi_h : A h \otimes  A \rightarrow A$ given by \ $a \cdot h \otimes  b \mapsto a \cdot \dt (b) + \{ a , b \}_J - a \cdot \{ 1 , b \}_J$, where $\dt$ satisfies \eqref{very important 2}, is a right $(A , A h \otimes  A)$-connection on $A$.  Note that  $\varphi_h$ is well-defined by \eqref{well}. 
Let $\sum  a_i \cdot h \otimes  b_i  \in A h \otimes  A$, then
\begin{align*}
 \varphi_h & \left( \sum  c \cdot a_i \cdot h \otimes  b_i \right)   = \sum c \cdot a_i \cdot \dt (b_i) + \sum  \{ c \cdot a_i , b_i \}_J - \sum  c \cdot a_i \cdot \{ 1 , b_i \}_J \\
 & = c \cdot \sum a_i \cdot \dt ( b_i ) + c \cdot \sum \{ a_i , b_i \}_J - \sum a_i \cdot \{ b_i , c \}_J \\
 & \phantom{={}}  - \sum c \cdot a_i \cdot \{ 1 , b_i \}_J - \sum c \cdot a_i \cdot \{ 1 , b_i \}_J\\
& = c \cdot \varphi_h \left( \sum  a_i \cdot h \otimes  b_i \right) - \rho_{ A h \otimes  A} \left( \sum a_i \cdot h \otimes  b_i \right) ( c ) 
\end{align*}
so $\varphi_{ h}$ satisfies \eqref{equation: right connection character axi}. Lastly, the expression for the curvature in  \eqref{equation: curv second} follows  directly from Lemma \ref{theorem: A-linear}. 
\end{proof}

\begin{remark}
\label{ann} 
Assume right $(A , Ah \otimes A)$-connections exist on $A$, and let $\mathcal D$ be the (non-empty) set of maps $\dt: A \rightarrow A$ satisfying \eqref{very important 2} for all $a \in A$. Since  maps $\dt_1  , \dt_2  \in \mathcal D$  satisfy $$r  \cdot \dt_1 (a) - r \cdot \dt_2 (a) = \{a , r \}_J - \{ a , r \}_J = 0 \quad \textnormal{for all $a \in A$},$$  we deduce  $\dt_1 (a) - \dt_2 (a)  \in \bigcap_{r \in \mathrm{Ann}_A ( \{ h \} ) } \mathrm{Ann}_A ( r) =: H$, hence it follows that   the set  $\mathcal D$   is an affine space over $\mathrm{Lin}_k ( A , H )$. 

\medskip

Furthermore, from  \eqref{j connection} we see that right  $(Ah \otimes A, A)$-connections  on $A$ are determined by maps $\dt \in \mathcal D$ so that given two connections $\varphi_h$ and $\varphi_h'$ we have $r \cdot ( \varphi_h - \varphi_h') (a) = 0$ for all $a \in A$.  As before,  we deduce that the set of right  $(A , Ah \otimes A)$-connections on $A$ is an affine space over $\mathrm{Lin}_k ( A , H)$. 

\end{remark}

\begin{proof}[Proof of Theorem \ref{main}] We now prove Parts 2 (a) and 2 (b):
\medskip

 To prove Part 2 (a), take $ \xi = r \otimes  1$ where $r \in \mathrm{Ann}_A ( \{ h \} )$, $\zeta = 1 \otimes  1 $ and $a = r$ in Lemma \ref{fundamental}, so we  deduce that  if there exists a right $( A , A h \otimes  A )$-connection on $A$, not necessarily flat, there  must exist  some $b \in A$ such that $r \cdot b = \{ 1 , r \}_J$, proving the claim. Note that  taking $a = 1$ in \eqref{very important 2} yields  $ 0  = r \cdot \dt \left( 1 \right) - \{ 1 , r \}_J$ for all $r  \in \mathrm{Ann}_A ( \{ h \} )$.

\medskip

For all $a \in A$, we denote by   $S_a $ the set $ \{ s \in A \mid r \cdot s  =   \{ a , r \}_J , \forall r \in \mathrm{Ann}_A ( \{ h\} ) \}$ of solutions of \eqref{very important 2}. 

\medskip

To prove  Part 2 (b), assume there exist right $(A , Ah \otimes A)$-connections on  $A$ which are flat, i.e.,  $\mathcal D$ is non-empty, so that \eqref{very important 2} has solutions $\dt (a) \in S_a \subset A$ for all $a \in A$ and furthermore,  $\Cr \left( h \otimes  f  , h \otimes  g  \right) = 0$ for all $f,  g  \in A$. Let $b \in A$ satisfy  $ \{ 1 , b \}_J = 0$, then by  \eqref{equation: curv second} we have
\begin{equation}
\label{C 1 y}
0 =  \Cr \left( h \otimes  1 , h \otimes  b \right) = - \{ 1 ,   \dt (b)  \}_J - \{  \dt (1)   , b \}_J .
\end{equation}
Hence if there exists no $c  \in A$ satisfying  $ \{ 1 , c \}_J =  \{ b ,  s  \}_J  $ for some $ s \in S_1$,  then there exists no such  map  $ \dt : A \rightarrow A$, and  hence there exists no flat right $(A, A h \otimes  A )$-connection on $A$. 
\end{proof}

\begin{corollary}
If $\textnormal{Ann}_A (h) = 0$,  the map $\dt: A \rightarrow A$ given by $a \mapsto \{ 1 , a \}_J$ satisfies \eqref{very important 2} and  hence induces the right  $( A ,  Ah \otimes A)$-connection   $\varphi : Ah \otimes A \rightarrow A$ on $A$ given by $a \cdot h \otimes b \mapsto \{ a , b \}_J$ which is shown to be flat, by a straightforward computation using \eqref{equation: curv second}.
\end{corollary}

\section{Lie--Rinehart algebras $(A, A h \otimes  A )$ with no antipode}

\label{main proof}

This section is dedicated to provide  examples of Lie--Rinehart algebras $(A, A h \otimes  A )$, constructed as in the previous section. The first of our examples admits no right $(A, Ah \otimes A)$-connections on $A$ while the second one does admit them, although none of them can be flat. 

\subsection{Example with no $(A, Ah \otimes A)$-connections on $A$}

  Let $A = k [ x , y ] / \langle x \cdot y ,  x^2 , y^2 \rangle$, let  $E \in \mathrm{Der}_k (A)$ be a derivation with $E (x) = y$, $ E (y) = 0$, and let $A$ be endowed with the Jacobi bracket  $ \{ a , b \}_J = a \cdot E (b) - E (a) \cdot  b$.

\medskip

Take  $h = y $, then   $\mathrm{Ann}_A ( \{ y \} ) = \mathrm{Span}_k \{ x , y \}$ and we have  $r  \cdot \{ \bullet  , \bullet \}_J = 0$ for all $r \in \mathrm{Ann}_A ( \{ y \} ) $.  Then, by Theorem \ref{main} Part 1,  the pair  $(A, A y \otimes  A  )$ is a Lie--Rinehart algebra with Lie bracket on $A y \otimes  A $  given by  $  \lbrack    y \otimes  f ,    y \otimes  g \rbrack_{ A y \otimes  A }   =  y \otimes  \{ f , g \}_J $   and anchor $\rho_{ A y \otimes  A }: A y \otimes  A  \rightarrow \mathrm{Der}_k (A)$ given by $ y \otimes  a  \mapsto \Phi_a = a  \cdot E ( \bullet ) $.  
Since $x \in \mathrm{Ann}_A ( \{ y \})$ and  there exists no $a \in A$ satisfying $ a \cdot x - \{ 1 , x \}_J = a \cdot x - y  = 0$ we deduce by Theorem \ref{main} Part 2 (a) that the Lie--Rinehart algebra $(A , A y \otimes  A )$ does not admit right $(A,Ay \otimes A)$-connections on $A$. Hence its universal enveloping algebra does not admit an antipode.

\medskip

Alternatively, we can prove this result by noting that $(Ay \otimes  A , [ \bullet , \bullet ]_{Ay \otimes  A})$ is isomorphic to the Heisenberg Lie algebra $H_3(k)$  of dimension $3$, with central element  $y \otimes  y$.  Let $\{ \alpha_1 , \alpha_2 , \alpha_3 \}$ be a basis for $H_3(k)$ with central element $\alpha_3$. Following \cite[Proposition 3.1]{KR1}, we define an $A$-module structure on $H_3(k)$ by $a \cdot \alpha_i := \chi (a) \alpha_i$,  $1 \leq i \leq 3$ where $\chi: A \rightarrow k$ is a character on $A$ given by $\chi ( x) = \chi (y) = 0$, and an anchor map $\rho: H_3(k) \rightarrow \mathrm{Der}_k (A)$ given by $\rho(\alpha_1) = E$, $\rho (\alpha_2) = \rho ( \alpha_3) = 0$. Then,  $(A, H_3 (k))$ is a Lie--Rinehart algebra  isomorphic to $(A, Ay \otimes  A)$. A similar argument as in the proof of  \cite[Theorem 1.1]{KR1} yields that there exist no right $(A, H_3(k))$-connections on $A$.

\subsection{Example with non--flat $(A, Ah \otimes A)$-connections on $A$}

Let $k= \mathbb Z_2$ and let $A = \mathbb Z_2 [ x, y , z ] / \langle x^4, y^6, z^2,  x \cdot y^4 , x^3 \cdot y , x^3 \cdot z    \rangle$. A basis for $A$ as $\mathbb Z_2$-module is:
\begin{equation}
\begin{multlined}\label{basis}
1, x , x^2 , x^3,  y, y^2, y^3, y^4, y^5,  z ,  x \cdot y, x^2 \cdot y , x \cdot y^2, x \cdot y^3 ,  x^2 \cdot y^2, x^2 \cdot y^3 ,  x  \cdot z ,  \\ x^2 \cdot z   ,    y \cdot z  ,    y^2 \cdot z , y^3 \cdot z , y^4 \cdot z , y^5 \cdot z , x \cdot y \cdot z , x^2 \cdot y \cdot z , x \cdot y^2 \cdot z ,  x \cdot y^3 \cdot z , \\  x^2 \cdot y^2 \cdot z, x^2 \cdot y^3 \cdot z .
\end{multlined}
\end{equation}
 Let    $E, F \in \mathrm{Der}_R (A)$ be derivations with $ E (x) = E (z) = x^2$ , $E (y)=0$ and $ F (x) = F(z) = 0$, $ F ( y ) =  z$. 
Then the images of  $E, F \in \mathrm{Der}_R (A)$  characterized in terms of the basis  for $A$  in   \eqref{basis} are: 
\begin{equation}
\begin{multlined}
\label{vanish}
 \mathrm{Im} ( E )  = \mathrm{Span}_{ \mathbb Z_2} \{  x^2  ,  x^2 \cdot  y, x^2 \cdot y^2, x^2 \cdot y^3,  x^2 \cdot z +   x^3  , \\  x^2 \cdot y \cdot z , x^2 \cdot y^2 \cdot z , x^2 \cdot y^3 \cdot z   \} \subset Ax^2 
\end{multlined}
\end{equation}
and
\begin{equation}
\mathrm{Im} (F)  = \mathrm{Span}_{ \mathbb Z_2} \{ z, y^2 \cdot z , y^4 \cdot z ,    x \cdot z, x^2 \cdot z , x \cdot y^2 \cdot z , x^2 \cdot y^2 \cdot z    \}
\end{equation}
 Hence we deduce that  
$$
\mathrm{Im} ( F \circ E )  = \mathrm{Span}_{ \mathbb Z_2} \{  x^2 \cdot z , x^2 \cdot y^2 \cdot z   \} \subset Ax^2 \cdot z $$
and 
\begin{equation*}
 \mathrm{Im} ( E \circ F )  = \mathrm{Span}_{ \mathbb Z_2} \{ x^2 , x^2 \cdot y^2 ,  x^2 \cdot z + x^3   , x^2 \cdot y^2 \cdot z   \} \subset Ax^2 
\end{equation*}
so   $ E (a) \cdot F ( E (b) ) \in A$, $E (a) \cdot E ( F(b) ) \in A$ vanish for all $a , b , c \in A$. 
So $A$ admits  the  Jacobi bracket  
\begin{equation*}
\{ a , b \}_J = E (a) \cdot F (b) - F (a) \cdot E (b) + a \cdot E (b) - E (a) \cdot  b.
\end{equation*}  
To see this, we  note $\{ 1 , \bullet \}_J = E$; next we check that   $ \{ \bullet , \bullet  \}_J$ satisfies the Leibniz rule in   \eqref{equation: generalized leibniz}:
\begin{align*}
\{ a , b \cdot c \}_J & = E (a) \cdot F ( b \cdot c ) - F (a) \cdot E (b \cdot c ) +  a  \cdot E ( b \cdot c ) - E (a) \cdot  b \cdot c \\
 & = b \cdot \{ a , c \}_J  + c \cdot \{ a , b \}_J + \{ 1 , a \}_J \cdot b \cdot c;
\end{align*}
and finally we check that $\{ \bullet  , \bullet \}_J$ satisfies the Jacobi identity:
\begin{small}
\begin{align*}
&  \{ a , \{ b , c \}_J \}_J + c.p.   =   E(a) \cdot F ( E(b) \cdot F(c) - F(b) \cdot E(c) +  b  \cdot E(c) - E(b) \cdot c ) \\
 &  \phantom{={}}  - F(a) \cdot E( E(b) \cdot F(c) - F(b) \cdot E(c) + b  \cdot E(c) - E(b) \cdot  c ) \\
 &  \phantom{={}} +  a \cdot E( E(b) \cdot F(c) - F(b) \cdot E(c)+  b \cdot E(c) - E(b) \cdot  c ) \\
 & \phantom{={}}  - E(a) \cdot ( E(b) \cdot F(c) - F(b) \cdot E(c) +  b \cdot E(c) - E(b) \cdot c ) + c.p.  \\
  & = E(a) \cdot F( E(b) ) \cdot F(c)   - E(a) \cdot F(b) \cdot F( E(c) )   + E(a) \cdot b \cdot F( E(c) )  - E(a) \cdot c \cdot F ( E(b) )  \\
 & \phantom{={}}  - F (a) \cdot E(b) \cdot E( F(c)) + F(a) \cdot E( F(b) ) \cdot E(c) + a \cdot E(b) \cdot E( F(c) )  - a \cdot E ( F(b) ) \cdot E(c) + c.p. \\
 & \phantom{={}}  = 0
\end{align*}
\end{small} 
\vspace{-11pt}
\noindent Now, taking $h = y^2$, we characterize  $  \mathrm{Ann}_A ( \{ y^2 \}) $  in terms of the basis for $A$ given  in \eqref{basis} as
\begin{equation} 
\label{ann basis}
\begin{multlined}
 \mathrm{Ann}_A ( \{ y^2 \} )  = \mathrm{Span}_{\mathbb Z_2} \left\{ x^3  , y^4 , y^5 ,   x \cdot y^2 , x \cdot y^3 , x^2 \cdot y^2 ,   x^2 \cdot y^3 , \right. \\
\left. y^4 \cdot z , y^5 \cdot z ,    x \cdot y^2 \cdot z , x \cdot y^3 \cdot z ,  x^2 \cdot y^2 \cdot z , x^2 \cdot y^3 \cdot z \right\}
 \end{multlined}
\end{equation}
 so that $r \cdot \{ \bullet , \bullet \}_J = 0$ for all $r \in \mathrm{Ann}_A ( \{ y^2 \})$. 
Hence, we deduce that  $ ( A , A y^2 \otimes_{ \mathbb Z_2} A)$ is a Lie--Rinehart algebra   with anchor map $\rho_{  Ay^2 \otimes_{\mathbb Z_2} A} : A y^2 \otimes_{\mathbb Z_2} A \rightarrow \mathrm{Der}_R (A) $  given by $ y^2 \otimes_{\mathbb Z_2} a \mapsto \Phi_a $ and   Lie bracket on the $A$-module $Ay^2 \otimes_{ \mathbb Z_2} A$  given by   $\lbrack   y^2  \otimes_{ \mathbb Z_2} f,  y^2 \otimes_{ \mathbb Z_2} g  \rbrack_{ A y^2 \otimes_{ \mathbb Z_2} A }   = y^2 \otimes_{\mathbb Z_2 } \{ f , g \}_J  $. %

\medskip

We now show that right $(A, Ay^2  \otimes_{\mathbb Z_2} A)$-connections do exist on $(A , \{ \bullet , \bullet \}_J )$. A straightforward computation using the characterization of $\mathrm{Ann}_A ( \{ y^2 \} )$ given in \eqref{ann basis} shows that     $x \in A$ satisfies 
\begin{equation}
\label{D1} 
r \cdot x = \{ 1 , r \}_J 
\end{equation}
 for all $r  \in \mathrm{Ann}_A (\{ y^2 \})$.   It is now straightforward to check that the map     $\dt: A \rightarrow A$ given by $a \mapsto - x \cdot F ( a ) + x \cdot a $ for all $a \in A$, satisfies the condition given in \eqref{very important 2} since $E(a) \cdot ( F(r) - r ) = 0$ for all $a \in A$ and $ r \in \textnormal{Ann}_A ( \{h \})$ so that
\begin{align*}
r \cdot \dt (a) &= r  \cdot ( - x \cdot F ( a ) + x \cdot a ) = - E (r  ) \cdot F (a) + E(r ) \cdot a \\
& = E(a) \cdot F(r  ) - E(r ) \cdot F(a) + a \cdot E(r ) - r  \cdot E(a) = \{ a , r  \}_J 
\end{align*}
for all $r  \in \mathrm{Ann}_A ( \{ y^2 \})$.   Hence, by Lemma \ref{twist}, the map
\begin{equation}
\varphi: A y^2 \otimes_{ \mathbb Z_2} A \longrightarrow A, \quad a \cdot y^2 \otimes_{ \mathbb Z_2} b \longrightarrow - a \cdot x  \cdot F(b)  + a \cdot b \cdot x + \{ a , b \}_J - a \cdot \{ 1 , b \}_J
\end{equation}
is a right $(A, A y^2 \otimes_{\mathbb Z_2} A )$-connection character on $A$.

\medskip

We now prove that none of the right $(A, Ay^2  \otimes_{\mathbb Z_2} A)$-connections on $A$ is flat. First note    $$ H = \bigcap_{r \in \mathrm{Ann}_A (\{ y^2 \} )} \mathrm{Ann}_A (r) =  x^2 A + y^2 A $$
that is
\begin{align*}
H =   \mathrm{Span}_{\mathbb Z_2 }  (  & x^2 , x^3 , x^2 \cdot y, x^2 \cdot y^2, x^2 \cdot y^3, x^2 \cdot z ,  x^2 \cdot y \cdot z  , x^2 \cdot  y^2 \cdot z , \\
 &  x^2 \cdot y^3 \cdot z , y^2 ,  x \cdot y^2 ,  y^3, y^4, y^5, y^2\cdot z ,  x \cdot y^3 ,   x  \cdot y^2 \cdot z    , \\
  &   y^3 \cdot z  ,    y^4 \cdot z , y^5 \cdot z ,    x \cdot y^3 \cdot z             ).
\end{align*}
 Then, from Remark \ref{ann} and \eqref{D1} we deduce that the only solutions of    the equation $a \cdot r   = \{ 1 , r \}_J$, for all $r \in \mathrm{Ann}_A ( \{ y^2 \} )$,  are  elements $a = x +  \alpha$ for all $\alpha \in H $. Since  $\{ 1 , y \}_J = E \left( y \right) = 0$, we can take $a = x + \alpha$ and $b = y$ in Theorem \ref{main}, and  compute
\begin{align*}
\{ a  , b \}_J & =  \left\{   x + \alpha , y \right\}_J   =      x^2 \cdot  z - x^2 \cdot y      + E( \alpha ) \cdot F(y) - y \cdot E( \alpha )  \\
 & = x^2 \cdot z - x^2 \cdot y +  \lambda_1 \cdot  x^2 \cdot y^2 \cdot z + \lambda_2 \cdot  x^2 \cdot y^3 + \lambda_3 \cdot  x^2 \cdot y^3 \cdot z 
\end{align*}
for $\lambda_1, \lambda_2, \lambda_3 \in \mathbb Z_2$, which by  \eqref{vanish} is not in the image of $E$.  Hence  there exists no $c \in A$ satisfying $\{a , b \}_J = \{ 1 , c \}_J$ and by Theorem \ref{main} Part 2 (b),   we find  that $A$ is not a  right  $( A , A y^2 \otimes_{ \mathbb Z_2} A )$-module.

\bibliographystyle{alpha}

\end{document}